\documentclass[a4paper,twoside]{article}

\usepackage{trdeg}

\author{Gregor Kemper \\
  \normalsize Technische Universit\"at M\"unchen, Zentrum Mathematik - M11 \\
  \normalsize  Boltzmannstr. 3, 85\,748 Garching, Germany \\
   \normalsize {\tt kemper$@$ma.tum.de}}
 
\title{The Transcendence Degree over a Ring}

\date{September 6, 2011}

\begin{document}

\maketitle

\begin{abstract}
  For a finitely generated algebra over a field, the transcendence
  degree is known to be equal to the Krull dimension. The aim of this
  paper is to generalize this result to algebras over rings. A new
  definition of the transcendence degree of an algebra $A$ over a ring
  $R$ is given by calling elements of $A$ algebraically dependent if
  they satisfy an algebraic equation over $R$ whose trailing
  coefficient, with respect to some monomial ordering, is~$1$. The
  main result is that for a finitely generated algebra over a
  Noetherian Jacobson ring, the transcendence degree is equal to the
  Krull dimension.
\end{abstract}

\section*{Introduction}

The equality of Krull dimension and transcendence degree for a
finitely generated algebra over a field is one of the fundamental
results in commutative algebra. Various extensions of this result have
appeared in the literature. \mycite{Onoda:Yoshida} generalized the
result to subalgebras of a finitely generated algebra over a field.
\mycite{Tanimoto:03} showed that for a finitely generated domain $A$
over a field and a prime ideal $P \in \Spec(A)$, the local ring $A_P$
has a subfield $L$ such that the transcendence degree of $A_P$ over
$L$ equals $\dim(A_P)$. Some authors, among them \mycite{Giral:81} and
\mycite{Hamann:86}, proposed several notions of a transcendence degree
of an algebra over a ring and studied their behavior. In the
introduction, \citename{Hamann:86} wrote that her paper might better
be titled: ``Why there is no notion of transcendence degree over
arbitrary commutative rings.''

This paper aims to take up the challenge posed by this comment. We
give a new definition of the transcendence degree of an algebra $A$
over a ring $R$ by calling elements $a_1 \upto a_n \in A$
algebraically dependent if they satisfy an equation $f(a_1 \upto a_n)
= 0$, where~$f$ a polynomial with coefficients in $R$ such that the
trailing coefficient of~$f$, with respect to some monomial ordering,
is~$1$ (see \dref{dTrdeg}). If $R$ is a field, this definition
coincides with the usual one. But in other cases, the new
transcendence degree behaves in unexpected ways. For example, the
transcendence degree of a ring over itself is ``usually'' not zero. As
the main result, we prove that if $A$ is finitely generated and $R$ is
a Noetherian Jacobson ring, then the Krull dimension of $A$ is equal
to the transcendence degree of $A$ over $R$. In fact, this result
extends to the case that $A$ is contained in a finitely generated
$R$-algebra, and if $A = R$, the hypothesis that $R$ is a Jacobson
ring can be dropped (see \tref{tMain}). In the case $A = R$ the result
was already proved for the lexicographic monomial ordering by
\mycite{Coquand.Lombardi} (see \tref{tCL}).

The paper is organized as follows. The first section contains the new
definition of the transcendence degree, some examples, and the
statement of the main result (\tref{tMain}). This is proved in the
second section. We also show that the validity of \tref{tMain}
characterizes Jacobson rings (see \rref{rConverse}). The last section
is devoted to some applications and to the question whether the
transcendence degree depends of the choice of a monomial ordering. We
conjecture that is does not (Conjecture~\ref{Conjecture}), and prove
some special cases (see \tref{tOrdering}).

This work was inspired by reading the above-mentioned paper of
\mycite{Coquand.Lombardi}, who characterized the Krull dimension by
certain types of identities. Interpreting this in terms of the
lexicographic monomial ordering led to the new definition of the
transcendence degree and prompted the questions to what extent this
depends on the choice of the monomial ordering, and whether one can
also prove a ``relative'' version involving the transcendence degree
over a subring. In his bachelor thesis~[\citenumber{Baerligea}],
Christoph \citename{Baerligea} studied (among other things) the first
question and found no example where the transcendence degree depends
on the monomial ordering. I wish to thank Peter Heinig for bringing
\citename{Coquand.Lombardi}'s article to my attention.



\section{A new definition of the transcendence degree} \label{sDefi}

All rings in this paper are assumed to be commutative with an identity
element~$1$. If $R$ is a ring, an $R$-algebra is a ring $A$ together
with a ring homomorphism $R \to A$. We call an $R$-algebra
\df{subfinite} if it is a subalgebra of a finitely generated
$R$-algebra. By $\dim(R)$ we will always mean the Krull dimension of
$R$. We follow the convention that the zero ring $R = \{0\}$ has Krull
dimension~$-1$. It will be convenient to work with the polynomial ring
$R[x_1,x_2, \ldots]$ with infinitely many indeterminates over a ring
$R$, and to understand a \df{monomial ordering} as a total
ordering~``$\preceq$'' on the set of monomials of $R[x_1,x_2, \ldots]$
such that the conditions $1 \preceq s$ and $s t_1 \preceq s t_2$ hold
for all monomials $s,t_1$, and~$t_2$ with $t_1 \preceq t_2$. Clearly
any monomial ordering on a polynomial ring $R[x_1 \upto x_n]$ with
finitely many indeterminates can be extended to a monomial ordering in
the above sense.

The following notions of algebraic dependence and transcendence degree
over a ring generalize the corresponding notions over a field.

\begin{defi} \label{dTrdeg}
  Let $R$ be a ring.
  \begin{enumerate}
  \item Let~``$\preceq$'' be a monomial ordering. A nonzero polynomial
    $f \in R[x_1,x_2,\ldots]$ is called \df{submonic} with respect
    to~``$\preceq$'' if its trailing coefficient (i.e., the
    coefficient of the least monomial having nonzero coefficient)
    is~$1$.
    
    A polynomial is called \df{submonic} if there exists a monomial
    ordering with respect to which it is submonic.
  \item Let $A$ be an $R$-algebra. Elements $a_1 \upto a_n \in A$ are
    called \df{algebraically dependent} over $R$ if there exists a
    submonic polynomial $f \in R[x_1 \upto x_n]$ such that $f(a_1
    \upto a_n) = 0$. (Of course the homomorphism $R \to A$ is applied
    to the coefficients of~$f$ before evaluating at $a_1 \upto a_n$.)
    Otherwise, $a_1 \upto a_n$ are called \df{algebraically
      independent} over $R$.
    
    We can also restrict~$f$ to be submonic with respect to a
    specified monomial ordering~''$\preceq$'', in which case we speak
    of algebraic dependence or independence with respect
    to~''$\preceq$''.
  \item For an $R$-algebra $A$, the \df{transcendence degree} of $A$
    over $R$ is defined as
    \begin{multline*}
      \trdeg(A:R) := \\
      \sup\left\{\strut[3mm] n \in \NN \mid \text{there exist} \ a_1
        \upto a_n \in A \ \text{that are algebraically independent
          over R}\right\}.
    \end{multline*}
    If every $a \in A$ is algebraically dependent over $R$, we set
    $\trdeg(A:R) := 0$ in the case $A \ne \{0\}$ and $\trdeg(A:R) :=
    -1$ in the case $A = \{0\}$. We write $\trdeg(R) := \trdeg(R:R)$
    for the transcendence degree of $R$ over itself.
    
    If~''$\preceq$'' is a monomial ordering, we define
    $\trdeg_\preceq(A:R)$ by requiring algebraic independence with
    respect to~''$\preceq$''.
  \end{enumerate}
\end{defi}

\begin{ex} \label{exTrdeg}
  \begin{enumerate}
    \renewcommand{\theenumi}{\arabic{enumi}}
  \item \label{exTrdeg1} If $R$ is an integral domain, then the
    elements of $R$ that are algebraically dependent over $R$ are~$0$
    and the units of $R$.
  \item \label{exTrdeg2} If $R$ is a nonzero finite ring, then
    $\trdeg(R) = 0$ since for each $a \in R$ there exist nonnegative
    integers $m < n$ such that $a^m = a^n$.
  \item \label{exTrdeg3} The following example shows that the notion
    of algebraic dependence with respect to a monomial ordering
    depends on the chosen monomial ordering. Let $R = K[t_1,t_2]$ be a
    polynomial ring in two indeterminates and let $a = t_1$ and $b =
    t_1 t_2$. The relation $b - t_2 a = 0$ tells us that $a,b$ are
    algebraically dependent over $R$ with respect to the lexicographic
    monomial ordering with $x_1 > x_2$. On the other hand, algebraic
    dependence over $R$ with respect to the lexicographic ordering
    with $x_2 > x_1$ would mean that there exist $i,j \in \NN_0$ such
    that $a^i b^j = t_1^{i+j} t_2 ^j $ lies in the $R$-ideal
    \[
    \left(b^{j+1},a^{i+1} b^j\right)_R = \left((t_1
      t_2)^{j+1},t_1^{i+j+1} t_2^j\right)_R,
    \]
    which is not the case.
    
  \item \label{exTrdeg4} We consider $R = \ZZ$ and claim that
    $\trdeg(\ZZ) = 1$. Since $\ZZ$ has nonzero elements which are not
    units, we have $\trdeg(\ZZ) \ge 1$. We need to show that all pairs
    of integers $a,b \in \ZZ$ are algebraically dependent over $\ZZ$.
    We may assume~$a$ and~$b$ to be nonzero and write
    \[
    a = \pm \prod_{i=1}^r p_i^{d_i} \quad \text{and} \quad b = \pm
    \prod_{i=1}^r p_i^{e_i},
    \]
    where the $p_i$ are pairwise distinct prime numbers and $d_i,e_i
    \in \NN_0$. Choose $n \in \NN_0$ such that $n \ge d_i/e_i$ for
    all~$i$ with $e_i > 0$. Then
    \[
    \gcd(a,b^{n+1}) = \prod_{i=1}^r p_i^{\min\{d_i,(n+1) e_i\}} \quad
    \text{divides} \quad \prod_{i=1}^r p_i^{n e_i} = b^n,
    \]
    so there exist $c,d \in \ZZ$ such that $b^n = c a + d b^{n+1}$.
    Since $f = x_2^n - c x_1 - d x_2^{n+1}$ is submonic (with respect
    to the lexicographic ordering with $x_1 > x_2$), this shows that
    $a,b$ are algebraically dependent.
    
    Clearly this argument carries over to any principal ideal domain
    that is not a field. It is remarkable that although the
    transcendence degree is an algebraic invariant, the above
    calculation has a distinctly arithmetic flavor. \exend
  \end{enumerate}
  \renewcommand{\exend}{}
\end{ex}

It becomes clear from \exref{exTrdeg}\eqref{exTrdeg1} that sums of
algebraic elements need not be algebraic, and from~\eqref{exTrdeg4}
that the transcendence degree does not behave additively for towers of
ring extensions.

\mycite{Coquand.Lombardi} proved that for a ring $R$ and an integer $n
\in \NN$, the inequality $\dim(R) < n$ holds if and only if for all
$a_1 \upto a_n \in R$ there exist $m_1 \upto m_n \in \NN_0$ such that
\begin{equation} \label{eqCL}
  \prod_{i=1}^n a_i^{m_i} \in \left.\left(\strut a_j \cdot
      \smash{\prod_{i=1}^j} a_i^{m_i} \right| j = 1 \upto n\right)_R
\end{equation}
(also see \mycite[Exercise~6.8]{Kemper.Comalg}). Using \dref{dTrdeg}
and writing $\lex$ for the lexicographic ordering with $x_i > x_{i+1}$
for all~$i$, we can reformulate this result as follows.

\begin{theorem}[\mycite{Coquand.Lombardi}] \label{tCL}
  If $R$ is a ring, then
  \[
  \trdeg_{\lex}(R) = \dim(R).
  \]
\end{theorem}

The following is the main result of this paper.

\begin{theorem} \label{tMain}
  \begin{enumerate}
  \item \label{tMainA} If $R$ is a Noetherian ring, then
    \[
    \trdeg(R) = \dim(R).
    \]
  \item \label{tMainB} If $R$ is a Noetherian Jacobson ring and $A$ is
    a subfinite $R$-algebra, then
    \[
    \trdeg_{\lex}(A:R) = \dim(A).
    \]
  \item \label{tMainC} If $R$ and $A$ are as in~\eqref{tMainB} and $A$
    is Noetherian, then
    \[
    \trdeg(A:R) = \dim(A).
    \]
  \end{enumerate}
\end{theorem}

Parts~\eqref{tMainB} and~\eqref{tMainC} generalize the classical
result that the Krull dimension of a finitely generated algebra over a
field is equal to its transcendence degree.

\section{Proof of the main result} \label{sProof}

The proof of \tref{tMain} is subdivided into various steps, which
contain results that are themselves of some interest.

Recall that in a Noetherian ring $R$ with $\dim(R) \ge n$ there exist
elements $a_1 \upto a_n$ and a prime ideal $P$ of height~$n$ that lies
minimally over $(a_1 \upto a_n)_R$ (see~[\citenumber{Kemper.Comalg},
Theorem~7.8]). Therefore the following theorem implies the inequality
\begin{equation} \label{eqLower}
  \dim(R) \le \trdeg(R).
\end{equation}

\begin{theorem} \label{tLower}
  Let $R$ be a Noetherian ring. If $a_1 \upto a_n \in R$ are elements
  such that $R$ has a prime ideal of height~$n$ lying minimally over
  $(a_1 \upto a_n)_R$, then the $a_i$ are algebraically independent
  over $R$.
\end{theorem}

\begin{proof}
  Let $P \in \Spec(R)$ be a prime ideal of height~$n$ lying minimally
  over $(a_1 \upto a_n)_R$. Clearly it suffices to show that the
  images of the $a_i$ in the localization $R_P$ are algebraically
  independent.  Substituting $R$ by $R_P$, we may therefore assume
  that $R$ is a local ring and $a_1 \upto a_n$ form a system of
  parameters. With $\mathfrak{q} := (a_1 \upto a_n)_R$, this implies
  that for all $j \in \NN_0$ the module $R/\mathfrak{q}^j$ has finite
  length, and for sufficiently large~$j$ this length is given by a
  polynomial of degree~$n$ in~$j$ (see
  \mycite[Theorem~17]{Matsumura:80}).
  
  By way of contradiction, assume that $a_1 \upto a_n$ are
  algebraically dependent. This means that there exists a monomial $t
  := \prod_{i=1}^n x_i^{d_i}$ and further monomials $t_k :=
  \prod_{i=1}^n x_i^{e_{k,i}}$ that are greater than~$t$ with respect
  to some monomial ordering such that
  \begin{equation} \label{eqAde}
    \prod_{i=1}^n a_i^{d_i} = \sum_{k=1}^m r_k \prod_{i=1}^n
    a_i^{e_{k,i}},
  \end{equation}
  where the $r_k$ are elements of $R$. By~[\citenumber{Kemper.Comalg},
  Exercise~9.2(b)], there exist positive integers $w_1 \upto w_n$ such
  that $\sum_{i=1}^n w_i d_i < \sum_{i=1}^n w_i e_{k,i}$ holds for $k
  \in \{1 \upto m\}$. (The exercise uses the so-called convex cone of
  the monomial ordering, and a solution is provided in the book.)
  Writing $(\underline{d}) = (d_1 \upto d_n)$ and $w(\underline{d}) :=
  \sum_{i=1}^n w_i d_i$, we can express the inequalities as
  \begin{equation} \label{eqW}
    w(\underline{d}) < w(\underline{e}_k) \quad (k \in \{1 \upto m\}).
  \end{equation}
  For $j \in \NN_0$ we define the ideal
  \[
  I_j := \left(\left.\smash{\prod_{i=1}^n} \strut a_i^{e_i} \right|
    e_1 \upto e_n \in \NN_0, \ w(\underline{e}) \ge j\right)_R
  \subseteq R.
  \]
  With $\overline{w}:=\max\{w_1 \upto w_n\}$ we have $I_{\overline{w}
    j} \subseteq \mathfrak{q}^j$ for all~$j$, so for~$j$ sufficiently
  large, the length of $R/I_{\overline{w} j}$ is bounded below by a
  polynomial of degree~$n$ in~$j$.
  
  Clearly $I_{j+1} \subseteq I_j$. We write
  \[
  m_j := \left|\left\{(e_1 \upto e_n) \in \NN_0^n \mid \strut
      w(\underline{e}) = j\right\}\right| \quad (j \in\ZZ)
  \]
  and claim that $I_j/I_{j+1}$ is generated as an $R$-module by $m_j -
  m_{j - w(\underline{d})}$ elements. In fact, $I_j/I_{j+1}$ is
  clearly generated by all $\prod_{i=1}^n a_i^{e_i}$ with
  $w(\underline{e}) = j$. But if $(\underline{e}) \in \NN_0^n$ has the
  form $(\underline{e}) = (\underline{d}) + (\underline{e}')$ with
  $w(\underline{e}') = j - w(\underline{d})$, then
  \[
  \prod_{i=1}^n a_i^{e_i} = \prod_{i=1}^n a_i^{d_i + e_i'}
  \underset{\eqref{eqAde}}{=} \sum_{k=1}^m r_k \prod_{i=1}^n
  a_i^{e_{k,i} + e_i'} \in I_{j+1},
  \]
  since
  \[
  w(\underline{e}_k + \underline{e}') = w(\underline{e}_k) +
  w(\underline{e}')\underset{\eqref{eqW}}{>} w(\underline{d}) +
  w(\underline{e}') = j.
  \]
  So we can exclude such a product $\prod_{i=1}^n a_i^{e_i}$ from our
  set of generators and are left with $m_j - m_{j - w(\underline{d})}$
  generators, as claimed. The definition of $I_j$ implies that
  $\mathfrak{q} I_j \subseteq I_{j+1}$, so $I_j/I_{j+1}$ is an
  $R/\mathfrak{q}$-module. With $l_0 := \length(R/\mathfrak{q})$, we
  obtain
  \[
  \length(R/I_j) = \sum_{i=0}^{j-1} \length\left(I_i/I_{i+1}\right)
  \le l_0 \sum_{i=0}^{j-1} \left(m_i - m_{i- w(\underline{d})}\right)
  = l_0 \!\!\! \sum_{i = j - w(\underline{d})}^{j-1} m_i.
  \]
  In the ring $\ZZ[[t]]$ of formal power series over $\ZZ$, we have
  \begin{align*}
    \sum_{j=0}^\infty \left(l_0 \!\!\! \strut\smash{\sum_{i = j -
          w(\underline{d})}^{j-1} \!\!\! m_i}\right) t^j & = l_0
    \sum_{i=0}^\infty \sum_{j=i+1}^{i+w(\underline{d})} m_i t^j = l_0
    \left(\strut\smash{\sum_{i=0}^\infty m_i t^i}\right)
    \left(\strut\smash{\sum_{j=1}^{w(\underline{d})} t^j}\right) \\
    & = \frac{l_0 \left(t + t^2 + \cdots +
        t^{w(\underline{d})}\right)}{(1-t^{w_1}) \cdots (1-t^{w_n})} =
    \frac{g(t)}{(1-t^{w_0})^n} = g(t) \cdot \sum_{j=0}^\infty \binom{j
      + n - 1}{n - 1} t^{w_0 j},
  \end{align*}
  where $w_0 := \lcm(w_1 \upto w_n)$ and $g(t) \in \ZZ[t]$. It follows
  that $\length(R/I_j)$ is bounded above by a polynomial of degree at
  most $n-1$ in~$j$, contradicting the fact that
  $\length(R/I_{\overline{w} j})$ is bounded below by a polynomial of
  degree~$n$ for large~$j$.  This contradiction finishes the proof.
\end{proof}

\begin{rem*}
  The converse statement of \tref{tLower} may fail: For example, the
  only prime ideal lying minimally over the class of~$x$ in $R :=
  K[x,y]/(x \cdot y)$ (with $K$ a field and~$x$ and~$y$
  indeterminates) has height~$0$, but the class of~$x$ is nevertheless
  algebraically independent over $R$.
\end{rem*}

\begin{proof}[Proof of \tref{tMain}\eqref{tMainA}]
  For any monomial ordering~``$\preceq$'', it follows directly from
  \dref{dTrdeg} that $\trdeg(R) \le \trdeg_\preceq(R)$. Applying this
  to the lexicographic ordering and using~\eqref{eqLower} and
  \tref{tCL} yields \tref{tMain}\eqref{tMainA}.
\end{proof}

Let $A$ be an algebra over a ring $R$ and let~''$\preceq$'' be a
monomial ordering. Then the inequalities
\begin{equation} \label{eqChain}
  \trdeg(A) \le \trdeg(A:R) \le \trdeg_\preceq(A:R)
\end{equation}
follow directly from \dref{dTrdeg}. The next goal is to prove
$\trdeg_{\lex}(A:R) \le \dim(A)$ in the case that $R$ is a Noetherian
Jacobson ring and $A$ is finitely generated over $R$. To achieve this
goal, we need four lemmas.  The proof of part~\eqref{lNgoA} of the
following lemma was shown to me by Viet-Trung Ngo.

\begin{lemma} \label{lNgo}
  Let $R$ be a Noetherian ring and $P \in \Spec(R)$.
  \begin{enumerate}
  \item \label{lNgoA} If $Q \in \Spec(R)$ such that $P \subseteq Q$ and
    $Q/P \in \Spec(R/P)$ has height at least~2, then there exist
    infinitely many prime ideals $P' \in \Spec(R)$ such that $P
    \subseteq P' \subseteq Q$ and $\height(P'/P) = 1$.
  \item \label{lNgoB} If $\mathcal{M} \subseteq \Spec(R)$ is an
    infinite set of prime ideals such that every $P' \in \mathcal{M}$
    satisfies $P \subseteq P'$ and $\height(P'/P) = 1$, then $P =
    \bigcap_{P' \in \mathcal{M}} P'$.
  \end{enumerate}
\end{lemma}

\begin{proof}
  \begin{enumerate}
  \item[\eqref{lNgoA}] By factoring out $P$ and localizing at $Q$ we
    may assume that $R$ is a local domain with maximal ideal $Q$, and
    $P = \{0\}$. For $a \in Q \setminus \{0\}$ there exists $P' \in
    \Spec(R)$ which is minimal over $(a)_R$. By the principal ideal
    theorem, $P'$ has height one. So
    \[
    Q \subseteq \bigcup
    \begin{Sb}
      P' \in \Spec(R), \\
      \height(P') = 1
    \end{Sb}
    P'.
    \]
    If there existed only finitely many $P' \in \Spec(R)$ of height
    one, it would follow by the prime avoidance lemma
    (see~[\citenumber{Kemper.Comalg}, Lemma~7.7]) that $Q$ is
    contained in one of them, contradicting the hypothesis $\height(Q)
    > 1$
  \item[\eqref{lNgoB}] Clearly $P \subseteq \bigcap_{P' \in
      \mathcal{M}} P' =: I$. If $P \subsetneqq I$, then every $P' \in
    \mathcal{M}$ would be a minimal prime ideal over $I$ and so
    $\mathcal{M}$ would be finite (see~[\citenumber{Kemper.Comalg},
    Corollary~3.14(d)]). \qed
  \end{enumerate}
  \renewcommand{\qed}{}
\end{proof}


\begin{lemma} \label{lJacobson}
  A Noetherian ring $R$ is a Jacobson ring if and only if for every $P
  \in \Spec(R)$ with $\dim(R/P) = 1$ there exist infinitely many
  maximal ideals $\mathfrak{m} \in \Spec(R)$ with $P \subseteq
  \mathfrak{m}$.
\end{lemma}

\begin{proof}
  We prove that the negations of both statements are equivalent. First
  assume that $R$ is not Jacobson. Choose $P \in \Spec(R)$ to be
  maximal among the prime ideals that are not intersections of maximal
  ideals. Then $P$ itself is not maximal, so $\dim(R/P) \ge 1$. On the
  other hand, it is impossible that $\dim(R/P) > 1$, since by
  \lref{lNgo} that would imply that $P$ is the intersection of
  strictly larger prime ideals and therefore (by its maximality) of
  maximal ideals. So $\dim(R/P) = 1$. Therefore every maximal ideal
  $\mathfrak{m}$ containing $P$ satisfies $\height(\mathfrak{m}/P) =
  1$, so by \lref{lNgo}\eqref{lNgoB} only finitely many
  such~$\mathfrak{m}$ exist.

  Conversely, assume that $R$ has a prime ideal $P$ with $\dim(R/P) =
  1$ such that only finitely many maximal ideals, say $\mathfrak{m}_1
  \upto \mathfrak{m}_n$, contain $P$. If $P = \mathfrak{m}_1 \cap
  \cdots \cap \mathfrak{m}_n$, then $P$ would contain at least one of
  the $\mathfrak{m}_i$, so $P$ would be maximal. Since this is not the
  case, $R$ is not Jacobson.
\end{proof}

The following lemma may be surprising since it does not require the
ring $R$ to be Jacobson.

\begin{lemma} \label{lBoundary}
  Let~$a$ be an element of a Noetherian ring $R$ and set
  \[
  U_a := \left\{a^n (1 + a x) \mid n \in \NN_0, \ x \in R\right\}
  \]
  Then the localization $U_a^{-1} R$ is a Jacobson ring.
\end{lemma}

\begin{proof}
  We will use the criterion from \lref{lJacobson} and the
  inclusion-preserving bijection between the prime ideals in $S :=
  U_a^{-1} R$ and the prime ideals $P \in \Spec(R)$ satisfying $U_a
  \cap P = \emptyset$ (see~[\citenumber{Kemper.Comalg}, Theorem~6.5]).
  Let $P \in \Spec(R)$ with $U_a \cap P = \emptyset$ such that
  $\dim\left(S/U_a^{-1} P\right) = 1$. Then there exists $P_1 \in
  \Spec(R)$ with $P \subsetneqq P_1$ and $U_a \cap P_1 = \emptyset$.
  The latter condition implies $a \notin P_1$ and $1 \notin P_1 +
  (a)_R$, so there exists $Q \in \Spec(R)$ such that $P_1 + (a)_R
  \subseteq Q$. It follows that $\height(Q/P) > 1$, so by
  \lref{lNgo}\eqref{lNgoA}, the set
  \[
  \mathcal{M} := \left\{P' \in \Spec(R) \mid P \subseteq P' \subseteq
    Q, \ \height(P'/P) = 1\right\}
  \]
  is infinite. Assume that the subset $\mathcal{M}' := \{P' \in
  \mathcal{M} \mid a \in P'\}$ is also infinite. Then
  \lref{lNgo}\eqref{lNgoB} would imply $P = \bigcap_{P' \in
    \mathcal{M}'} P'$, so $a \in P$, contradicting $U_a \cap P =
  \emptyset$. We conclude that $\mathcal{M} \setminus \mathcal{M}'$ is
  infinite. Let $P' \in \mathcal{M} \setminus \mathcal{M}'$. Then $P'
  \subseteq Q$ and $a \in Q$ imply that $1 + a x \notin P'$ for every
  $x \in R$, so $U_a \cap P' = \emptyset$. Therefore $U_a^{-1} P' \in
  \Spec(S)$, and we also have $U_a^{-1} P \subsetneqq U_a^{-1} P'$.
  Since $\dim\left(S/U_a^{-1} P\right) = 1$, this implies that
  $U_a^{-1} P'$ is a maximal ideal. So by \lref{lJacobson}, the
  infinity of $\mathcal{M} \setminus \mathcal{M}'$ implies that $S$ is
  a Jacobson ring.
\end{proof}

\begin{rem*}
  The localization $U_a^{-1} R$ from \lref{lBoundary} was also used by
  \mycite{Coquand.Lombardi}. They called it the {\em boundary} of~$a$
  in $R$.
\end{rem*}

\begin{lemma} \label{lField}
  Let $R$ be a Jacobson ring and let $A$ be a finitely generated
  $R$-algebra that is a field. Then the kernel of the map $R \to A$ is
  a maximal ideal.
\end{lemma}

\begin{proof}
  We may assume that the map $R \to A$ is injective, so we may view
  $R$ as a subring of $A$. We need to show that $R$ is a field. Since
  $A$ is finitely generated as an algebra over $K := \Quot(R)$ (the
  field of fractions), it is algebraic over $K$
  (see~[\citenumber{Kemper.Comalg}, Lemma~1.1(b)]). So $A$ is a finite
  field extension of $K$, and choosing a basis yields a map
  \[
  \map{\phi}{A}{K^{n \times n}}
  \]
  sending each $a \in A$ to the representation matrix of the linear
  map given by multiplication by~$a$. Let $a_1 \upto a_n$ be
  generators of $A$ as an $R$-algebra and choose $b \in R \setminus
  \{0\}$ to be a common denominator of all matrix entries of all
  $\phi(a_i)$. Then $\phi(a_i) \in R[b^{-1}]^{n \times n}$, and
  since~$\phi$ is a homomorphism of $R$-algebras, its image is
  contained in $R[b^{-1}]^{n \times n}$. If $a \in R \setminus \{0\}$,
  then $a^{-1} \in A$, so
  \[
  \diag(a^{-1} \upto a^{-1}) = \phi(a^{-1}) \in R[b^{-1}]^{n \times
    n}.
  \]
  This implies $a^{-1} \in R[b^{-1}]$, so $R[b^{-1}]$ is a field. From
  this it follows by \mycite[Lemma~4.20]{eis} that $R$ is a field.
\end{proof}

As announced, we can now prove the upper bound for the transcendence
degree. As above, $\lex$ stands for the lexicographic ordering with
$x_i > x_{i+1}$ for all~$i$.

\begin{prop} \label{pUpper}
  \begin{enumerate}
  \item \label{pUpperA} If $R$ is a ring, then
    \[
    \trdeg_{\lex}(R) \le \dim(R).
    \]
  \item \label{pUpperB} If $A$ is a finitely generated algebra over a
    Noetherian Jacobson ring $R$, then
    \[
    \trdeg_{\lex}(A:R) \le \dim(A).
    \]
  \end{enumerate}
\end{prop}

\begin{rem*}
  Part~\eqref{pUpperA} is contained in \citename{Coquand.Lombardi}'s
  result (\tref{tCL}). We include a proof of this part for the
  reader's convenience.
\end{rem*}

\begin{proof}[Proof of \pref{pUpper}]
  We prove both parts simultaneously, setting $A = R$ in the case of
  part~\eqref{pUpperA}. It is clear that without loss of generality we
  may assume the map $R \to A$ to be injective, so we may view $R$ as
  a subring of $A$. We may also assume $A \ne \{0\}$ and $\dim(A) <
  \infty$. We use induction on $n := \dim(A) + 1$.
  
  Let $a_1 \upto a_n \in A$. Consider the multiplicative set
  \[
  U := \left\{f(a_n) \mid f \in R[x] \ \text{is submonic}\right\}
  \subseteq A
  \]
  and set $A' := U^{-1} A$. By way of contradiction, assume $\dim(A')
  \ge n - 1$. Using the correspondence between prime ideals in $A'$
  and prime ideals in $A$ that do not intersect with $U$, we obtain a
  chain $P_1 \subsetneqq P_2 \subsetneqq \cdots \subsetneqq P_n$ with
  $P_i \in \Spec(A)$ and $U \cap P_i = \emptyset$. Since $\dim(A) = n
  - 1$, $A/P_n$ must be a field. In the case of part~\eqref{pUpperB},
  it follows by \lref{lField} that $R \cap P_n \subseteq R$ is a
  maximal ideal. In the case of part~\eqref{pUpperA}, this is also
  true since $R = A$. So $A/P_n$ is an algebraic field extension of
  $R/(R \cap P_n)$ (see~[\citenumber{Kemper.Comalg}, Lemma~1.1(b)]).
  From $U \cap P_n = \emptyset$ we conclude that $a_n + P_n \in A/P_n$
  is invertible. So there exists $g \in R[x]$ such that $a_n g(a_n) -
  1 \in P_n$.  But $1 - x g$ is submonic, so $1 - a_n g(a_n) \in U$,
  contradicting $U \cap P_n = \emptyset$.
  
  We conclude that $\dim(A') + 1 < n-1$. If $A' = \{0\}$ (which must
  happen if $n = 1$), then $0 \in U$, so $a_n$ satisfies a submonic
  equation and therefore $a_1 \upto a_n$ are algebraically dependent
  with respect to~$\lex$. Having dealt with this case, we may assume
  $A' \ne \{0\}$.
  
  Clearly $R' := U^{-1} R[a_n] \subseteq A'$. In the case of
  part~\eqref{pUpperA} we have $R' = A'$. In the case of
  part~\eqref{pUpperB}, $A'$ is finitely generated as an $R'$-algebra,
  and by \lref{lBoundary}, $R'$ is a Jacobson ring. So in both cases
  the induction hypothesis tells us that $\frac{a_1}{1} \upto
  \frac{a_{n-1}}{1} \in A'$ satisfy a polynomial $\tilde{f} \in R'[x_1
  \upto x_{n-1}]$ that is submonic with respect to~$\lex$. Multiplying
  the coefficients of~$\tilde{f}$ by a suitable element from $U$, we
  obtain $\widehat{f} \in R[a_n][x_1 \upto x_{n-1}]$ whose trailing
  coefficient $c_t \in R[a_n]$ lies $U$ such that $\widehat{f}(a_1
  \upto a_{n-1}) = 0$. By replacing every coefficient $c \in R[a_n]$
  of~$\widehat{f}$ by a $c' \in R[x_n]$ with $c'(a_n) = c$, we obtain
  $f \in R[x_1 \upto x_n]$ with $f(a_1 \upto a_n) = 0$.  Since $c_t
  \in U$, we may choose the coefficient $c_t' \in R[x_n]$ of~$f$ to be
  submonic. The trailing coefficient of~$f$ (with respect to~$\lex$)
  is equal to the trailing coefficient of $c_t'$, which is~$1$.
  Therefore~$f$ is submonic with respect to~$\lex$.
  
  We conclude that $a_1 \upto a_n$ are algebraically dependent with
  respect to~$\lex$. Since they were chosen as arbitrary elements of
  $A$, this shows that $\trdeg_{\lex}(A:R) \le n - 1 = \dim(A)$.
\end{proof}

We can now finish the proof of \tref{tMain}.

\begin{proof}[Proof of \tref{tMain}\eqref{tMainB} and~\eqref{tMainC}]
  \tref{tCL} and \dref{dTrdeg} imply
  \[
  \dim(A) = \trdeg_{\lex}(A) \le \trdeg_{\lex}(A:R).
  \]
  If $A$ is Noetherian, \eqref{eqLower} and~\eqref{eqChain} yield the
  finer inequality
  \[
  \dim(A) \le \trdeg(A) \le \trdeg(A:R) \le \trdeg_{\lex}(A:R).
  \]
  So for the proof of~\eqref{tMainB} and~\eqref{tMainC} it suffices to
  show that $\trdeg_{\lex}(A:R) \le \dim(A)$. We may assume that $0
  \le \dim(A) < \infty$ and write $n = \dim(A) + 1$.
  
  By hypothesis, $A$ is a subalgebra of a finitely generated
  $R$-algebra $B$. Let $P_1 \upto P_r \in \Spec(B)$ be the minimal
  prime ideals of $B$, and assume that we can show $\trdeg_{\lex}(A/A
  \cap P_i:R) \le \dim(A/A \cap P_i)$ for all~$i$. Then for $a_1 \upto
  a_n \in A$ there exist polynomials $f_1 \upto f_r \in R[x_1 \upto
  x_n]$ that are submonic with respect to $\lex$ such that $f_i(a_1
  \upto a_n) \in P_i$. So $\prod_{i=1}^r f_i(a_1 \upto a_n)$ lies in
  the nilradical of $B$, hence there exists~$k$ such that $a_1 \upto
  a_n$ satisfy the polynomial $\prod_{i=1}^r f_i^k$, which is also
  submonic with respect to $\lex$. This shows that we may assume $B$
  to be an integral domain.
  
  By \mycite[Proposition~2.1(b)]{Giral:81}
  (or~[\citenumber{Kemper.Comalg}, Exercise~10.3]), there exists a
  nonzero $a \in A$ such that $A[a^{-1}]$ is finitely generated as an
  $R$-algebra. So we may apply \pref{pUpper}\eqref{pUpperB} and get
  $\trdeg_{\lex}\left(A[a^{-1}]:R\right) \le
  \dim\left(A[a^{-1}]\right)$. We obtain
  \[
  \trdeg_{\lex}(A:R) \le \trdeg_{\lex}\left(A[a^{-1}]:R\right) \le
  \dim\left(A[a^{-1}]\right) \le \dim(A),
  \]
  where the first inequality follows directly from \dref{dTrdeg}, and
  the last since $A[a^{-1}]$ is a localization of $A$. This completes
  the proof.
\end{proof}

\begin{rem} \label{rConverse}
  \begin{enumerate}
  \item \label{rConverseA} The hypothesis that $R$ be a Jacobson ring
    cannot be dropped from \tref{tMain}\eqref{tMainB}
    and~\eqref{tMainC}. In fact, the validity of
    \tref{tMain}\eqref{tMainB} and~\eqref{tMainC} characterizes
    Jacobson rings in the following sense: If $R$ is a non-Jacobson
    ring, then by \mycite[Lemma~4.20]{eis}, $R$ has a nonmaximal prime
    ideal $P$ such $S := R/P$ contains a nonzero element~$b$ for which
    $A := S[b^{-1}]$ is a field. So $\dim(A) = 0$, but~$b$ is not a
    unit in $S$, so \exref{exTrdeg}\eqref{exTrdeg1} yields
    \[
    1 \le \trdeg(S) = \trdeg(S:R) \le \trdeg(A:R) \le
    \trdeg_{\lex}(A:R).
    \]
    Since $A$ is a finitely generated $R$-algebra, the assertions of
    \tref{tMain}\eqref{tMainB} and~\eqref{tMainC} fail for $R$.
  \item \label{rConverseB} Neither can the hypothesis that $A$ is
    subfinite be dropped. In fact, if $R$ is any nonzero ring, we can
    choose a maximal ideal $\mathfrak{m}$ of $R$ and form the
    polynomial ring $S := \left(R/\mathfrak{m}\right)[x]$ and $A :=
    \Quot(S)$. Then
    \[
    \dim(A) = 0 < 1 = \trdeg(A:R/\mathfrak{m}) =
    \trdeg(A:R) = \trdeg_{\lex}(A:R).
    \]
  \end{enumerate}
  \renewcommand{\remend}{}
\end{rem}

\section{Some applications and further results} \label{sApplications}

It seems to be rare that \tref{tMain} helps to compute the dimension
of rings. In fact, the transcendence degree seems to be the less
accessible quantity in most cases, so knowledge of the dimension
provides information about the structure of the ring that is encoded
in the transcendence degree. \exref{exTrdeg}\eqref{exTrdeg4} contains
such an instance. Here is a further example.

\begin{ex} \label{exNumber}
  Let~$a$ and~$b$ be two nonzero algebraic numbers (i.e., elements of
  an algebraic closure of $\QQ$). There exists $d \in \ZZ \setminus
  \{0\}$ such that~$a$ and~$b$ are integral over $\ZZ[d^{-1}]$, so $A
  := \ZZ\left[a,b,d^{-1}\right]$ has Krull dimension~$1$. By
  \tref{tMain}\eqref{tMainB}, $\trdeg_{\lex}(A:\ZZ) = 1$, so $a,b$
  satisfy a polynomial $f \in \ZZ[x_1,x_2]$ that is submonic with
  respect to $\lex$. If $x_1^m x_2^n$ is the trailing monomial of~$f$,
  then all monomials of~$f$ are divisible by $x_1^m$, so we may assume
  $m = 0$.  We obtain
  \[
  b^n = a \cdot g(a,b) + b^{n+1} \cdot h(a,b)
  \]
  with $g,h \in \ZZ[x_1,x_2]$ polynomials. It is not so clear how the
  existence of such a relation follows directly from the properties of
  algebraic numbers.
\end{ex}


The following corollary of \tref{tMain}\eqref{tMainB} may be new:

\begin{cor} \label{cAB}
  Let $R$ be a Noetherian Jacobson ring, $B$ a subfinite $R$-algebra,
  and $A \subseteq B$ a subalgebra. Then
  \[
  \dim(A) \le \dim(B).
  \]
\end{cor}

\begin{ex} \label{exInvariant}
  Let $R$ be a ring that is finitely generated as a $\ZZ$-algebra, $G
  \subseteq \Aut(R)$ a group of automorphisms of $R$ and $H \subseteq
  G$ a subgroup. Then \cref{cAB} tells us that
  \[
  \dim\left(R^G\right) \le \dim\left(R^H\right),
  \]
  even though the invariant rings need not be finitely generated (see
  \mycite{nag:a}).
\end{ex}

We also get the following geometric-topological consequence.

\begin{theorem} \label{tScheme}
  Let $X$ be a scheme of finite type over a Noetherian Jacobson ring
  $R$. Let $Y$ be a locally closed subset of the underlying
  topological space of $X$. Then $\dim(Y) = \dim(\overline{Y})$.
\end{theorem}

\begin{proof}
  By part~\eqref{lTopB} of the following \lref{lTop}, we need to show
  that if $\dim(\overline{Y}) \ge n$ for an integer~$n$, then $\dim(Y)
  \ge n$. (We will invoke \lref{lTop}\eqref{lTopB} numerous times
  during this proof without always mentioning it.) By
  \lref{lTop}\eqref{lTopD}, there exists an open affine subset $U$ of
  $X$ such that $\dim(U \cap \overline{Y}) \ge n$. By
  \lref{lTop}\eqref{lTopC}, $U \cap \overline{Y}$ is the closure of $U
  \cap Y$ in $U$. It is also clear that $U \cap Y$ is locally closed
  in $U$. Moreover, if $U = \Spec(A)$, then $A$ is finitely generated
  as an $R$-algebra (see \mycite[Chapter~II, Exercise~3.3(c)]{hart}).
  So by substituting $X$ by $U$ and $Y$ by $U \cap Y$, we may assume
  that $X = \Spec(A)$ with $A$ a finitely generated $R$-algebra.
  
  Since $\overline{Y} = \Spec(B)$ with $B$ a quotient ring of $A$, we
  may substitute $A$ by $B$ and assume that $Y$ is dense in $X$. $X$
  has an irreducible component $X_i$ with $\dim(X_i) \ge n$. Since
  $X_i \cap Y$ is nonempty, $X_i \cap Y$ is dense in $X_i$. So by
  substituting $X$ by $X_i$ and factoring out the nilradical of $A$,
  we may assume $A$ to be an integral domain. Since $Y \subseteq X$ is
  nonempty and open, there exists a nonzero ideal $I \subseteq A$ such
  that
  \[
  Y = \left\{P \in \Spec(A) \mid I \not\subseteq P\right\}.
  \]
  Choose $0 \ne a \in I$. Then
  \[
  D_a := \left\{P \in \Spec(A) \mid a \notin P\right\} \subseteq Y,
  \]
  so it suffices to show that $\dim(D_a) \ge n$. But $D_a$ is
  homeomorphic to $\Spec\left(A\left[a^{-1}\right]\right)$
  (see~[\citenumber{Kemper.Comalg}, Theorem~6.5 and Exercise~6.5]).
  Since $A$ is a Jacobson ring (see \mycite[Theorem~4.19]{eis}),
  \cref{cAB} yields
  \[
  \dim\left(A\left[a^{-1}\right]\right) \ge \dim(A),
  \]
  so $\dim(Y) \ge \dim(D_a) = \dim\left(A\left[a^{-1}\right]\right)
  \ge n$.
\end{proof}

The following lemma was used in the previous proof.

\begin{lemma} \label{lTop}
  Let $X$ be a topological space and $Y \subseteq X$ a subset equipped
  with the subspace topology.
  \begin{enumerate}
  \item \label{lTopA} $Y$ is irreducible if and only if its closure
    $\overline{Y}$ is irreducible.
  \item \label{lTopB} The inequality $\dim(Y) \le \dim(X)$ holds for
    the dimensions of $X$ and $Y$ as topological spaces.
  \item \label{lTopC} If $U \subseteq X$ is an open subset, then $U
    \cap \overline{Y}$ is the closure of $U \cap Y$ in $U$.
  \item \label{lTopD} If $\mathcal{A}$ is a set of open subsets of $X$
    such that $X = \bigcup_{U \in \mathcal{A}} U$ and if $\dim(Y) \ge
    n$ holds for an integer~$n$, then there exists $U \in \mathcal{A}$
    such that $\dim(U \cap Y) \ge n$.
  \end{enumerate}
\end{lemma}

\begin{proof}
  The proofs of~\eqref{lTopA} and~\eqref{lTopB} are straightforward
  and left to the reader.
  
  To prove~\eqref{lTopC}, let $Z \subseteq X$ be closed with $U \cap Y
  \subseteq Z$.  Then $\overline{Y} \subseteq \left(\overline{Y}
    \setminus U\right) \cup Z$, so $U \cap \overline{Y} \subseteq U
  \cap Z$. This shows that $U \cap \overline{Y}$ is the smallest
  subset of $X$ that contains $U \cap Y$ and is closed in $U$.
  
  Under the hypothesis of~\eqref{lTopD}, we may assume $Y = X$ since
  $\left\{Y \cap U \mid U \in \mathcal{A}\right\}$ is an open covering
  of $Y$. There exists a chain $X_0 \subsetneqq X_1 \subsetneqq \cdots
  \subsetneqq X_n$ of closed irreducible subsets of $X$. We can choose
  $U \in \mathcal{A}$ with $U \cap X_0 \ne \emptyset$. Then $U_i := U
  \cap X_i \ne \emptyset$, and the $U_i$ form an ascending chain of
  closed subsets in $U$. Since $U_i$ is nonempty and open in $X_i$,
  its closure equals $X_i$. This implies that the inclusions between
  the $U_i$ are strict, and, by~\eqref{lTopA}, that the $U_i$ are
  irreducible.  Therefore $\dim(U) \ge n$.
\end{proof}

\begin{rem*}
  Jacobson rings are characterized by the validity of \tref{tScheme}.
  Indeed, every non-Jacobson ring $R$ has a nonmaximal prime ideal $P$
  such $S := R/P$ contains a nonzero element~$b$ for which $S[b^{-1}]$
  is a field (see \rref{rConverse}\eqref{rConverseA}). Then with $X :=
  \Spec(S)$ und $Y := \{Q \in X \mid b \notin Q\}$ we have $\dim(Y) =
  0$, but $\overline{Y} = X$ and $\dim(X) > 0$.
\end{rem*}

It seems odd that the lexicographic ordering plays such a special role
in Theorems~\ref{tCL} and~\ref{tMain}. Can it be substituted by other
monomial orderings? \tref{tOrdering} below gives answers in some
special cases. We need some preparations for its proof.

Let us call a monomial ordering ``$\preceq$'' \df{weight-graded} if
there exist positive real numbers $w_1, w_2,\ldots$ such that if
$\prod_{i=1}^n x_i^{d_i} \preceq \prod_{i=1}^n x_i^{e_i}$ (with~$n$,
$d_i$, $e_i \in \NN_0$), then $\sum_{i=1}^n w_i d_i \le \sum_{i=1}^n
w_i e_i$.

\begin{lemma} \label{lWeighted}
  Let $a_1 \upto a_n$ be elements of a Noetherian ring $R$ such that
  $\dim(R) < n$ and $\dim\left(R/(a_1 \upto a_n)_R\right) \le 0$. Then
  $a_1 \upto a_n$ are algebraically dependent with respect to every
  weight-graded monomial ordering.
\end{lemma}

\begin{proof}
  Let ``$\preceq$'' be a weight-graded monomial ordering. Consider the
  set $J \subseteq R$ of all trailing monomials of polynomials $f \in
  R[x_1 \upto x_n]$ with $f(a_1 \upto a_n) = 0$. It is easy to see
  that $J$ is an ideal and $I := (a_1 \upto a_n)_R \subseteq J$. We
  need to show that $J = R$. Suppose that we can show that for all
  maximal ideals $\mathfrak{m} \in \Spec(R)$ with $I \subseteq
  \mathfrak{m}$, the elements $\frac{a_i}{1} \in R_{\mathfrak{m}}$ are
  algebraically dependent with respect to ``$\preceq$''. Then $J \cap
  (R \setminus \mathfrak{m})$ is nonempty, so $J = R$. This shows that
  we can assume that $R$ is a local ring and $I$ is contained in its
  maximal ideal.
  
  By hypothesis, $R/I$ is Artinian, so by
  \mycite[Theorem~17]{Matsumura:80}, the length of $R/I^j$ is given by
  a polynomial of degree $\dim(R) < n$ for~$j$ large enough.  Since
  multiplying all weights $w_i$ by the same positive constant does not
  change the monomial ordering, we may assume that $w_i \ge 1$ for $i
  \in \{1 \upto n\}$. For $(\underline{e}) = (e_1 \upto e_n) \in
  \NN_0^n$ write $w(\underline{e}) := \sum_{i=1}^n w_i e_i$, and for
  $j \in \NN_0$ set
  \[
  I_j := \left(\left.\strut\smash{\prod_{i=1}^n} a_i^{e_i} \right|
    (\underline{e}) \in \NN_0^n,\ w(\underline{e}) \ge j\right)_R.
  \]
  Then $I^j \subseteq I_j$, so the length of $R/I_j$ is bounded above
  by a polynomial of degree $< n$. Take $j \in \NN_0$ and consider the
  set
  \[
  A_j := \left\{(\underline{d}) \in \NN_0^n \mid j \le
    w(\underline{d}) < j + 1\right\}.
  \]
  By way of contradiction, assume that the $a_i$ are algebraically
  independent with respect to ``$\preceq$''. For $(\underline{d}) \in
  A_j$, this assumption and the fact that ``$\preceq$'' is
  weight-graded imply
  \[
  I_{j+1} \subseteq \left(\left.\strut\smash{\prod_{i=1}^n} a_i^{e_i}
    \right| \smash{\prod_{i=1}^n} x_i^{e_i} \succ
    \smash{\prod_{i=1}^n} x_i^{d_i}\right)_R \subsetneqq
  \left(\left.\strut\smash{\prod_{i=1}^n} a_i^{e_i} \right|
    \smash{\prod_{i=1}^n} x_i^{e_i} \succeq \smash{\prod_{i=1}^n}
    x_i^{d_i}\right)_R \subseteq I_j.
  \]
  So ordering $A_j$ according to ``$\preceq$'' yields a chain of
  length $|A_j|$ of ideals between $I_{j+1}$ and $I_j$. This implies
  $\length\left(I_j/I_{j+1}\right) \ge |A_j|$, so
  $\length\left(R/I_j\right) \ge \left|\left\{(\underline{d}) \in
      \NN_0^n \mid w(\underline{d}) < j\right\}\right|$. With
  $\overline{w} := \max\{w_1 \upto w_n\}$, we obtain
  \[
  \length\left(R/I_j\right) \ge \left|\left\{(\underline{d}) \in
      \NN_0^n \mid \overline{w}(d_1 + \cdots + d_n) < j\right\}\right|
  = \binom{\lceil j/\overline{w} \rceil - 1 + n}{n},
  \]
  contradicting the fact that the length of $R/I_j$ is bounded above
  by a polynomial of degree $< n$. This contradiction finishes the
  proof.
\end{proof}

\begin{theorem} \label{tOrdering}
  Let $A$ be a Noetherian ring and let ``$\preceq$'' be a monomial
  ordering.
  \begin{enumerate}
  \item \label{tOrderingA} If $A$ is an algebra over a ring $R$ that
    contains a field $K$ such that $A$ is a subfinite $K$-algebra,
    then
    \[
    \trdeg_\preceq(A:R) = \dim(A).
    \]
  \item \label{tOrderingB} If $A = B_P$ with $B$ a finitely generated
    algebra over a field and $P \in \Spec(B)$, then
    \[
    \trdeg_\preceq(A) = \dim(A).
    \]
  \item \label{tOrderingC} If $\dim(A) \le 1$, then
    \[
    \trdeg_\preceq(A) = \dim(A).
    \]
  \item \label{tOrderingD} If $\dim(A) \le 1$ and $A$ is an subfinite
    algebra over a Noetherian Jacobson ring $R$, then
    \[
    \trdeg_\preceq(A:R) = \dim(A).
    \]
  \end{enumerate}
\end{theorem}

\begin{proof}
  By \dref{dTrdeg} we may assume $A \ne \{0\}$. By~\eqref{eqLower} and
  \dref{dTrdeg} we have
  \[
  \dim(A) \le \trdeg(A) \le \trdeg_\preceq(A) \le \trdeg_\preceq(A:R),
  \]
  where we set $R = A$ in the case of~\eqref{tOrderingB}
  and~\eqref{tOrderingC}. So we may assume $n := \dim(A) + 1 < \infty$
  and need to show that all $a_1 \upto a_n \in A$ are algebraically
  dependent over $R$ with respect to ``$\preceq$''.
  \begin{enumerate}
  \item[\eqref{tOrderingA}] By \tref{tMain}\eqref{tMainC}, the $a_i$
    are algebraically dependent over $K$. Since $K$ is a field, the
    algebraic dependence is with respect to~``$\preceq$'', and since
    $K \subseteq R$, it is over $R$.
  \item[\eqref{tOrderingB}] Let $P_1 \upto P_r \in \Spec(A)$ be the
    minimal prime ideals in $A$. For each~$i$, $A/P_i$ is the
    localization of a finitely generated domain over a field at a
    prime ideal, so by \mycite{Tanimoto:03}, $A/P_i$ contains a field
    $K_i$ such that $\trdeg\left(A/P_i:K_i\right) =
    \dim\left(A/P_i\right)$. Therefore
    \[
    \trdeg_\preceq\left(A/P_i\right) \le
    \trdeg_\preceq\left(A/P_i:K_i\right) =
    \trdeg\left(A/P_i:K_i\right) = \dim\left(A/P_i\right) \le \dim(A).
    \]
    So we obtain polynomials $f_i \in A[x_1 \upto x_n]$ that are
    submonic with respect to ``$\preceq$'' such that $f_i(a_1 \upto
    a_n) \in P_i$. A power of the product of the $f_i$ yields a
    submonic equation for $a_1 \upto a_n$.
  \item[\eqref{tOrderingC}] If $\dim(A) = 0$, the algebraic dependence
    of the $a_i$ follows from \tref{tCL} since every monomial ordering
    restricts to the lexicographic ordering on $A[x_1]$. So we may
    assume $\dim(A) = 1$. It follows by \mycite{Robbiano:85} that the
    restriction of ``$\preceq$'' to $A[x_1,x_2]$ is either a
    lexicographic ordering or weight-graded.  By \tref{tCL} we may
    assume the latter.
  
    Let $P_1 \upto P_r$ be the minimal prime ideals of $A$ satisfying
    $a_1 \notin P_i$ and set $I := \bigcap_{i=1}^r P_i$ (with $I := R$
    in the case $r = 0$). Every $P \in \Spec(A)$ with $I \subseteq P$
    contains at least one of the $P_i$, and if $a_1 \in P$, then $P_i
    \subsetneqq P$, so $\dim(A/P) = 0$. This shows that
    $\dim\left(A/(I + A a_1)\right) \le 0$. Applying \lref{lWeighted}
    to $R := A/I$ yields $f \in A[x_1,x_2]$ that is submonic with
    respect to ``$\preceq$'' such that $f(a_1,a_2) \in I$. By
    multiplying~$f$ by~$x_1$, we may assume $f(a_1,a_2) \in A a_1$, so
    by the definition of $I$, $f(a_1,a_2)$ lies in every minimal prime
    ideal of $A$.  Therefore a suitable power of~$f$ yields a submonic
    equation for $a_1,a_2$.
  \item[\eqref{tOrderingD}] By \cref{cAB}, the subalgebra $A'
    \subseteq A$ generated by the $a_i$ has dimension at most~$1$.
    By~\eqref{tOrderingC}, the $a_i$ satisfy an equation $f \in A'[x_1
    \upto x_n]$ that is submonic with respect to~``$\preceq$''. Obtain
    $\widehat{f} \in R[x_1 \upto x_n]$ from~$f$ by replacing every
    coefficient $c \in A'$ of~$f$ be a $\widehat{c} \in R[x_1 \upto
    x_n]$ with $\widehat{c}(a_1 \upto a_n) = c$, where the trailing
    coefficient of~$f$ is replaced by $1 \in R$. It follows
    that~$\widehat{f}$ is submonic with respect to~``$\preceq$'' and
    $\widehat{f}(a_1 \upto a_n) = 0$. \qed
  \end{enumerate}
  \renewcommand{\qed}{}
\end{proof}

In view of \tref{tOrdering}, a candidate that comes to mind for a ring
$R$ such that $\trdeg_\preceq(R) > \dim(R)$ for some monomial
ordering~``$\preceq$'' is the polynomial ring $\ZZ[x]$. Using a short
program written in MAGMA~[\citenumber{magma}], I tested millions of
randomly selected triples of polynomials from $\ZZ[x]$ and verified
that they were all algebraically dependent with respect to the graded
reverse lexicographic ordering, even over the subring $\ZZ$. This
prompts the following conjecture:

\begin{conjecture} \label{Conjecture}
  \tref{tMain}\eqref{tMainA} and~\eqref{tMainC} holds with
  ``$\trdeg$'' replaced by ``$\trdeg_\preceq$'', with~``$\preceq$'' an
  arbitrary monomial ordering.
\end{conjecture}

So far, all efforts to prove the conjecture have been futile. Let me
mention that it would follow if one could get rid of the hypothesis
``$\dim\left(R/(a_1 \upto a_n)_R\right) \le 0$'' in \lref{lWeighted}.

\end{document}